\documentclass{amsart}
\usepackage{amssymb}
\usepackage{amsmath}

\newtheorem{theorem}{Theorem}[section]
\newtheorem{lemma}[theorem]{Lemma}
\newtheorem{corollary}[theorem]{Corollary}
\theoremstyle{definition}

\theoremstyle{remark}

\numberwithin{equation}{section}

\newcommand{\bea}{\begin{eqnarray}}
\newcommand{\eea}{\end{eqnarray}}
\newcommand{\ba}{\begin{array}}
\newcommand{\ea}{\end{array}}

\input{tcilatex}

\begin{document}
\title[On Some Geometric And Topological Properties of Sequence Spaces]{On
Some Geometric Properties Of\\
Sequence Space Defined By\\
de la Vall\'{e}e-Poussin Mean}
\author{NEC\.{I}P \c{S}\.{I}M\c{S}EK}
\address{\.{I}STANBUL COMMERCE UNIVERSITY, Department of Mathematics, \"{U}sk%
\"{u}dar, \.{I}stanbul, TURKEY.}
\email{necsimsek@yahoo.com}

\begin{abstract}
In this work, we investigate $k-$ nearly uniform convex $\left( k-NUC\right) 
$ \ and the uniform Opial properties of the sequence space defined by de la
Vall\'{e}e-Poussin mean. Also we give some corollaries concerning the
geometrical properties of this space.
\end{abstract}

\maketitle

\section{Introduction}

In summability theory, de la Vall\'{e}e-Poussin's mean is first used to
define the $(V,\lambda )$-summability by Leindler \cite{Leindler-1965}.
Malkowsky and Sava\c{s} \cite{MalSav-2000} introduced and studied some
sequence spaces which arise from the notion of generalized de la Vall\'{e}%
e-Poussin mean. Also the $(V,\lambda )$-summable sequence spaces have been
studied by many authors including \cite{Mik2006} and \cite{Savsav-2003}.

In literature, there have been many papers on the geometrical properties of
Banach spaces. Some of them are as follows: In \cite{Opial}, Opial defined
the Opial property with his name mentioned and he proved that $\ell
_{p}(1<p<\infty )$ satisfies this property but the spaces $Lp$$[0,2\pi ]$ $%
(p\neq 2,$ $1<p<\infty )$ do not. Franchetti \cite{Franchetti}\ has shown
that any infinite dimensional Banach space has an equivalent norm satisfying
the Opial property. Later, Prus \cite{Prus} has introduced and investigated
uniform Opial property for Banach spaces. In \cite{Huff}, the notion of
nearly uniform convexity for Banach spaces was introduced by Huff. It is an
infinite dimensional counterpart of the classical uniform convexity. Also
Huff \cite{Huff} proved that every nearly uniformly convex Banach space is
reflexive and it has the uniformly Kadec-Klee property. However, Kutzarova 
\cite{Kutzarova} defined and studied k-nearly uniformly convex Banach spaces.

Recently, there has been a lot of interest in investigating geometric
properties of sequence spaces. Some of the recent work on sequence spaces
and their geometrical properties is given in the sequel: Shue \cite{Shue}
first defined the Ces\'{a}ro sequence spaces with a norm. In \cite%
{Cui-Hudzik}, it is shown that the Ces\'{a}ro sequence spaces $ces_{p}$ $%
\left( 1<p<\infty \right) $ have $k-$nearly uniform convex and uniform Opial
properties. \c{S}im\c{s}ek and Karakaya \cite{Nec-Vat} studied the uniform
Opial property and some other geometric properties\ of generalized modular
spaces of Ces\'{a}ro type defined by weighted means. In addition, some
related papers on this topic can be found in \cite{Basar},\cite{Chen},\cite%
{Vatan-2007},\cite{Liu-1996},\cite{Mur-2007},\cite{Mur-Rif-Mik},\cite%
{Mus-1983} and \cite{Ek-Vat-Nec}.

Quite recently, \c{S}im\c{s}ek et al \cite{Nec-Ek-Vat} introduced a new
sequence space defined by de la Vall\'{e}e-Poussin's mean and investigated
some geometric properties as Kadec-Klee and Banach-Saks of type $p$.
Moreover, the sequence space involving de la Vall\'{e}e-Poussin's mean is
more general than Ces\textit{\'{a}}ro sequence space defined by Shue \cite%
{Shue} and investigated by Cui and Hudzik \cite{Cui-Hud}.

The main purpose of this paper is to investigate \ uniform Opial property
and $k$-nearly uniformly convex property of the sequence space defined in 
\cite{Nec-Ek-Vat}. In addition it will be given some corollaries concerning
this space.

\section{Preliminaries and Notation}

\label{sec2}

Let $(X,||\cdot ||)$ (for the brevity $X=(X,||\cdot ||)$ ) be a normed
linear space and let $B(X)$ (resp. $S(X)$ ) be the closed unit ball (resp.
unit sphere) of $X$. The space of all real sequences is denoted by $\ell
^{0} $. For any sequence $\{x_{n}\}$ in $X$, we denote by $conv(\{x_{n}\})$
the convex hull of the elements of $\{x_{n}\}$ (see \cite{Bynum}).

A Banach space $X$ is called $uniformly$ $convex$ $(UC)$ if for each $%
\varepsilon >0$, there is $\delta >0$ such that for $x,y\in S(X)$, the
inequality $\left\Vert x-y\right\Vert >\varepsilon $ implies that 
\begin{equation*}
\left\Vert \frac{1}{2}(x+y)\right\Vert <1-\delta .
\end{equation*}

Recall that for a number $\varepsilon >0$ a sequence $\{x_{n}\}$ is said to
be an $\varepsilon -seperated$ $sequence$ if 
\begin{equation*}
sep\left( \{x_{n}\}\right) =\inf \{\left\Vert x_{n}-x_{m}\right\Vert ,\text{ 
}n\neq m\}>\varepsilon .
\end{equation*}

A Banach space $X$ is said to have the $Kadec-Klee$ $property$ $(H$ $%
property)$ if every weakly convergent sequence on the unit sphere is
convergent in norm.

A Banach space $X$ is said to have the $uniform$ $Kadec-Klee$ $property$ $%
(UKK)$ if for every $\varepsilon >0$ there exists $\delta >0$ such that if $%
x $ is the weak limit of a normalized $\varepsilon $-separated sequence,
then $\left\Vert x\right\Vert <1-\delta $ (see \cite{Huff}). We have that
every (UKK) Banach space have the Kadec-Klee property.

A Banach space $X$ is said to be the $nearly$ $uniformly$ $convex$ $(NUC)$
if for every $\varepsilon >0$ there exists $\delta >0$ such that for every
sequence $\{x_{n}\}\subset B(X)$ with $sep\left( \{x_{n}\}\right)
>\varepsilon $, we have 
\begin{equation*}
conv(\{x_{n}\})\cap (1-\delta )B(X)\neq \varnothing \text{.}
\end{equation*}

Let $k\geq 2$ be an integer. A Banach space $X$ is said to be $k-nearly$ $%
uniformly$ $convex$ $(k-NUC)$ if for any $\varepsilon >0$ there exists $%
\delta >0$ such that for every sequence $\{x_{n}\}\subset B(X)$ with $%
sep\left( \{x_{n}\}\right) >\varepsilon $, there are $n_{1},n_{2},...,n_{k}%
\in \mathbb{N}$ such that 
\begin{equation*}
\left\| \frac{x_{n_{1}}+x_{n_{2}}+...+x_{n_{k}}}{k}\right\| <1-\delta \text{.%
}
\end{equation*}

Of course a Banach space $X$ is $(NUC)$ whenever it is $(k-NUC)$ for some
integer $k\geq 2$. Clearly, $(k-NUC)$ Banach spaces are $(NUC)$ but the
opposite implication does not hold in general (see \cite{Kutzarova}).

A Banach space $X$ is said to have the $Opial$ $property$ if every sequence $%
\{x_{n}\}$ weakly convergent to $x_{0}$ satisfies 
\begin{equation*}
\underset{n\rightarrow \infty }{\lim \inf }\left\Vert x_{n}\right\Vert <%
\underset{n\rightarrow \infty }{\lim \inf }\left\Vert x_{n}+x\right\Vert
\end{equation*}%
for every $x\in X$ (see \cite{Opial}).

A Banach space $X$ is said to have the $uniform$ $Opial$ $property$ if every 
$\varepsilon >0$ there exists $\delta >0$ such that for each weakly null
sequence $\{x_{n}\}\subset S(X)$ and $x\in X$ with $||x||\geq \varepsilon $,
we have (see \cite{Prus}) 
\begin{equation*}
1+\tau \leq \underset{n\rightarrow \infty }{\lim \inf }\left\Vert
x_{n}+x\right\Vert \text{.}
\end{equation*}

A point $x\in S(X)$ is called an $extreme$ $point$ if for any $y,z\in B(X)$
the equality $2x=y+z$ implies $y=z$.

A Banach space $X$ is said to be $rotund$ (abbreviated as (abbreviated as $%
(R)$) if every point of $S(X)$ is an extreme point.

A Banach space $X$ is said to be $fully$ $k-rotund$ (write $kR$) (see \cite%
{Fan-Glick}) if for every sequence $\{x_{n}\}\subset B(X)$, 
\begin{equation*}
\left\Vert x_{n_{1}}+x_{n_{2}}+...+x_{n_{k}}\right\Vert \rightarrow k\text{
\ \ as \ \ }n_{1},n_{2},...,n_{k}\rightarrow \infty
\end{equation*}

\noindent implies that $\{x_{n}\}$ is convergent.

It is well known that $(UC)$ implies $(kR)$ and $(kR)$ implies $((k+1)R)$,
and $(kR)$ spaces are reflexive and rotund, and it is easy to see that $%
(k-NUC)$ implies $(kR)$.

In this paper, we will need the following inequalities in the sequel; 
\begin{equation*}
\left\vert a_{k}+b_{k}\right\vert ^{p}\leq 2^{p-1}\left( \left\vert
a_{k}\right\vert ^{p}+\left\vert b_{k}\right\vert ^{p}\right) ,
\end{equation*}%
for $p\geq 1.$

Let $\Lambda =(\lambda _{k})$ be a nondecreasing sequence of positive real
numbers tending to infinity and let $\lambda _{1}=1$ and $\lambda _{k+1}\leq
\lambda _{k}+1$.

The generalized de la Vall\'{e}e-Poussin means of a sequence $x=\{x_{k}\}$
are defined as follows: 
\begin{equation*}
t_{k}(x)=\frac{1}{\lambda _{k}}\dsum\limits_{j\in I_{k}}x_{k}\text{ \ \ \ \
\ \ where \ \ \ }I_{k}=[k-\lambda _{k}+1,k]\text{ \ \ \ \ for \ \ }k=1,2,...%
\text{ \ \ .}
\end{equation*}

We write 
\begin{equation*}
\lbrack V,\lambda ]_{0}=\left\{ x\in \ell ^{0}:\underset{k\rightarrow \infty 
}{\lim }\frac{1}{\lambda _{k}}\sum\limits_{j\in I_{k}}|x_{j}|=0\right\}
\end{equation*}
\begin{equation*}
\lbrack V,\lambda ]=\left\{ x\in \ell ^{0}:x-le\in \lbrack V,\lambda ]_{0},%
\text{ for some }l\in \mathbb{C}\right\}
\end{equation*}
and 
\begin{equation*}
\lbrack V,\lambda ]_{\infty }=\left\{ x\in \ell ^{0}:\underset{k}{\sup }%
\frac{1}{\lambda _{k}}\sum\limits_{j\in I_{k}}|x_{j}|<\infty \right\}
\end{equation*}
for the sequence spaces that are strongly summable to zero, strongly
summable and strongly bounded by the de la Vall\'{e}e-Poussin method,
resp.(see \cite{Leindler-1965}). In the special case where $\lambda _{k}=k$
for $k=1,2,...$ the spaces $[V,\lambda ]_{0},$ $[V,\lambda ]$ and $%
[V,\lambda ]_{\infty }$ reduce to the spaces $w_{0},$ $w$ and $w_{\infty }$
introduced by Maddox \cite{Maddox-1}.

The following new paranormed sequence space defined in \cite{Nec-Ek-Vat}.

\begin{equation*}
V(\lambda ;p)=\left\{ x=(x_{j})\in \ell ^{0}:\sum\limits_{k=1}^{\infty
}\left( \frac{1}{\lambda _{k}}\sum\limits_{j\in I_{k}}|x_{j}|\right)
^{p_{k}}<\infty \right\} .
\end{equation*}
If we take $p_{k}=p$ for all $k$; the space $V(\lambda ;p)$ reduced to
normed space $V_{p}(\lambda )$ defined by

\QTP{Body Math}
\begin{equation*}
V_{p}(\lambda )=\left\{ x=(x_{j})\in \ell ^{0}:\sum\limits_{k=1}^{\infty
}\left( \frac{1}{\lambda _{k}}\sum\limits_{j\in I_{k}}|x_{j}|\right)
^{p}<\infty \right\}.
\end{equation*}
The details of the sequence spaces mentioned above can be found in \cite%
{Nec-Ek-Vat}.

\section{Main Results}

\label{sec3}

In this section we show that the space $V_{p}(\lambda )$ is $(k-NUC)$ and
have uniform Opial property. Firstly we need an important lemma.

\begin{lemma}
\label{Lem 2.1} Let $X$ $\subset V_{p}(\lambda )$. For any $\varepsilon >0$
and $L>0$, there exists $\delta >0$ such that for all $x,y\in X$, 
\begin{equation*}
\left\vert \dsum\limits_{n=1}^{\infty }\left( \frac{1}{\lambda _{n}}%
\dsum\limits_{i\in I_{n}}\left\vert x(i)+y(i)\right\vert \right)
^{p}-\dsum\limits_{n=1}^{\infty }\left( \frac{1}{\lambda _{n}}%
\dsum\limits_{i\in I_{n}}\left\vert x(i)\right\vert \right) ^{p}\right\vert
<\varepsilon \text{,}
\end{equation*}%
whenever 
\begin{equation*}
\dsum\limits_{n=1}^{\infty }\left( \frac{1}{\lambda _{n}}\dsum\limits_{i\in
I_{n}}\left\vert x(i)\right\vert \right) ^{p}<L\text{ \ \ and }%
\dsum\limits_{n=1}^{\infty }\left( \frac{1}{\lambda _{n}}\dsum\limits_{i\in
I_{n}}\left\vert y(i)\right\vert \right) ^{p}\leq \delta .
\end{equation*}
\end{lemma}

\begin{proof}
(See \cite{Cui-Hudzik}) Let $||.||$ denote the norm in $V_{p}(\lambda )$.
Then $\left\| y\right\| ^{p}<\delta $ implies 
\begin{equation*}
{\LARGE |}\text{ }||x+y||-\left\| x\right\| \text{ }{\LARGE |}\leq \left\|
(x+y)-x\right\| =\left\| y\right\| \leq \delta ^{\frac{1}{p}}.
\end{equation*}

Since the function $g(t)=t^{p}$ is uniformly continuous on the interval $%
\left[ 0,L^{\frac{1}{p}}+1\right] $, we get the assertion of the lemma.
\end{proof}

\begin{theorem}
\label{Thm 2.2} The space $V_{p}(\lambda )$ is $(k-NUC)$ for any integer $%
k\geq 2$ where $(1<p<\infty ).$
\end{theorem}

\begin{proof}
Let $\varepsilon >0$ and $(x_{n})\subset B(V_{p}(\lambda ))$ with $sep\left(
\{x_{n}\}\right) >\varepsilon .$ Let $x_{n}^{m}=(0,0,...,x_{n}(m),$ $%
x_{n}(m+1),...)$ for each $m\in \mathbb{N}$. Since for each $i\in \mathbb{N}$%
, $\{x_{n}(i)\}_{i=1}^{\infty }$ is bounded therefore using the diagonal
method one can find a subsequence $\{x_{n_{k}}\}$ of $\{x_{n}\}$ such that
the sequence $\{x_{n_{k}}(i)\}$ converges for each $i\in \mathbb{N}$.

Therefore, there exists an increasing sequence of positive integer $(k_{m})$
such that $sep(\{x_{n_{k}}^{m}\}_{k>k_{m}})\geq \varepsilon $. Hence there
is a sequence of positive integers $(n_{m})_{m=1}^{\infty }$ with $%
n_{1}<n_{2}<n_{3}<...$ \ such that%
\begin{equation}
\left\Vert x_{n_{m}}^{m}\right\Vert \geq \frac{\varepsilon }{2}  \label{3.1}
\end{equation}
for all $m\in \mathbb{N}$.

Write $I_{p}(x)=\dsum\limits_{n=1}^{\infty }\left( \frac{1}{\lambda _{n}}%
\dsum\limits_{i\in I_{n}}\left\vert x(i)\right\vert \right) ^{p}$ and put $%
\varepsilon _{1}=\frac{k^{p-1}-1}{2k^{p}(k-1)}(\frac{\varepsilon }{2})^{p}$.
Then by Lemma \ref{Lem 2.1}, there exists $\delta >0$ such that%
\begin{equation}
\left\vert I_{p}(x+y)-I_{p}(x)\right\vert <\varepsilon _{1}  \label{3.2}
\end{equation}
whenever $I_{p}(x)\leq 1$ and $I_{p}(y)\leq \delta $ (see\cite{Cui-Pluc}).

There exists $m_{1}\in \mathbb{N}$ such that $I_{p}(x_{1}^{m_{1}})\leq
\delta $. Next there exists $m_{2}>m_{1}$ such that $I_{p}(x_{2}^{m_{2}})%
\leq \delta $. In such a way, there exists $m_{2}<m_{3}<...<m_{k-1}$ such
that $I_{p}(x_{j}^{m_{j}})\leq \delta $ for all $j=1,2,...,k-1$. Define $%
m_{k}=m_{k-1}+1$. By condition (\ref{3.1}), there exists $n_{k}\in \mathbb{N}
$ such that $I_{p}(x_{n_{k}}^{m_{k}})\geq \left( \frac{\varepsilon }{2}%
\right) ^{p}$. Put $n_{i}=i$ for $1\leq i\leq k-1$. Then in virtue of (\ref%
{3.1}), (\ref{3.2}) and convexity of the function $f(u)=|u|^{p}$, we get%
\begin{equation*}
I_{p}(\frac{x_{n_{1}}+x_{n_{2}}+...+x_{n_{k}}}{k})=\dsum%
\limits_{n=1}^{m_{1}}\left( \frac{1}{\lambda _{n}}\dsum\limits_{i\in
I_{n}}\left\vert \frac{x_{n_{1}}(i)+x_{n_{2}}(i)+...+x_{n_{k}}(i)}{k}%
\right\vert \right) ^{p}+\text{ \ \ \ \ \ \ \ \ \ \ \ \ \ \ \ \ \ \ \ \ \ \
\ \ \ }
\end{equation*}%
\begin{eqnarray*}
&&\text{ \ \ \ \ \ \ }+\dsum\limits_{n=m_{1}+1}^{\infty }\left( \frac{1}{%
\lambda _{n}}\dsum\limits_{i\in I_{n}}\left\vert \frac{%
x_{n_{1}}(i)+x_{n_{2}}(i)+...+x_{n_{k}}(i)}{k}\right\vert \right) ^{p} \\
\text{ \ \ \ \ \ \ \ \ \ } &\leq &\dsum\limits_{n=1}^{m_{1}}\frac{1}{k}%
\dsum\limits_{j=1}^{k}\left( \frac{1}{\lambda _{n}}\dsum\limits_{i\in
I_{n}}\left\vert x_{n_{j}}(i)\right\vert \right) ^{p} \\
&&\text{ \ \ }+\dsum\limits_{n=m_{1}+1}^{\infty }\left( \frac{1}{\lambda _{n}%
}\dsum\limits_{i\in I_{n}}\left\vert \frac{%
x_{n_{2}}(i)+x_{n_{3}}(i)+...+x_{n_{k}}(i)}{k}\right\vert \right)
^{p}+\varepsilon _{1} \\
&=&\dsum\limits_{n=1}^{m_{1}}\frac{1}{k}\dsum\limits_{j=1}^{k}\left( \frac{1%
}{\lambda _{n}}\dsum\limits_{i\in I_{n}}\left\vert x_{n_{j}}(i)\right\vert
\right) ^{p}+\dsum\limits_{n=m_{1}+1}^{m_{2}}\left( \frac{1}{\lambda _{n}}%
\dsum\limits_{i\in I_{n}}\left\vert \frac{%
x_{n_{2}}(i)+x_{n_{3}}(i)+...+x_{n_{k}}(i)}{k}\right\vert \right) ^{p}+ \\
&&+\dsum\limits_{n=m_{2}+1}^{m_{3}}\left( \frac{1}{\lambda _{n}}%
\dsum\limits_{i\in I_{n}}\left\vert \frac{%
x_{n_{2}}(i)+x_{n_{3}}(i)+...+x_{n_{k-1}}(i)+x_{n_{k}}(i)}{k}\right\vert
\right) ^{p}+\varepsilon _{1}
\end{eqnarray*}%
\begin{eqnarray*}
&& \\
&\leq &\dsum\limits_{n=1}^{m_{1}}\frac{1}{k}\dsum\limits_{j=1}^{k}\left( 
\frac{1}{\lambda _{n}}\dsum\limits_{i\in I_{n}}\left\vert
x_{n_{j}}(i)\right\vert \right) ^{p}+\dsum\limits_{n=m_{1}+1}^{m_{2}}\frac{1%
}{k}\dsum\limits_{j=2}^{k}\left( \frac{1}{\lambda _{n}}\dsum\limits_{i\in
I_{n}}\left\vert x_{n_{j}}(i)\right\vert \right) ^{p}+ \\
&&+\dsum\limits_{n=m_{2}+1}^{\infty }\left( \frac{1}{\lambda _{n}}%
\dsum\limits_{i\in I_{n}}\left\vert \frac{%
x_{n_{3}}(i)+x_{n_{4}}(i)+...+x_{n_{k-1}}(i)+x_{n_{k}}(i)}{k}\right\vert
\right) ^{p}+2\varepsilon _{1} \\
&&\cdots \\
&\leq &\frac{I_{p}\left( x_{n_{1}}\right) +...+I_{p}\left(
x_{n_{k-1}}\right) }{k}+\frac{1}{k}\dsum\limits_{n=1}^{m_{k-1}}\left( \frac{1%
}{\lambda _{n}}\dsum\limits_{i\in I_{n}}\left\vert x_{n_{k}}(i)\right\vert
\right) ^{p}+ \\
&&+\dsum\limits_{n=m_{k-1}+1}^{\infty }\left( \frac{1}{\lambda _{n}}\left(
\dsum\limits_{i\in I_{n}}\left\vert \frac{x_{n_{k}}(i)}{k}\right\vert
\right) \right) ^{p}+(k-1)\varepsilon _{1} \\
&\leq &\frac{k-1}{k}+\frac{1}{k}\dsum\limits_{n=1}^{m_{k-1}}\left( \frac{1}{%
\lambda _{n}}\dsum\limits_{i\in I_{n}}\left\vert x_{n_{k}}(i)\right\vert
\right) ^{p}+\frac{1}{k^{p}}\dsum\limits_{n=m_{k-1}+1}^{\infty }\left( \frac{%
1}{\lambda _{n}}\dsum\limits_{i\in I_{n}}\left\vert x_{n_{k}}(i)\right\vert
\right) ^{p}+(k-1)\varepsilon _{1} \\
&=&1+\frac{1}{k}\left( 1-\dsum\limits_{n=m_{k-1}+1}^{\infty }\left( \frac{1}{%
\lambda _{n}}\dsum\limits_{i\in I_{n}}\left\vert x_{n_{k}}(i)\right\vert
\right) ^{p}\right) +\frac{1}{k^{p}}\dsum\limits_{n=m_{k-1}+1}^{\infty
}\left( \frac{1}{\lambda _{n}}\dsum\limits_{i\in I_{n}}\left\vert
x_{n_{k}}(i)\right\vert \right) ^{p}+(k-1)\varepsilon _{1} \\
&\leq &1+(k-1)\varepsilon _{1}-\left( \frac{k^{p-1}-1}{k^{p}}\right)
\dsum\limits_{n=m_{k-1}+1}^{\infty }\left( \frac{1}{\lambda _{n}}%
\dsum\limits_{i\in I_{n}}\left\vert x_{n_{k}}(i)\right\vert \right) ^{p} \\
&\leq &1+(k-1)\varepsilon _{1}-\left( \frac{k^{p-1}-1}{k^{p}}\right) \left( 
\frac{\varepsilon }{2}\right) ^{p} \\
&=&1-\frac{1}{2}\left( \frac{k^{p-1}-1}{k^{p}}\right) \left( \frac{%
\varepsilon }{2}\right) ^{p}
\end{eqnarray*}%
Under the condition (\ref{3.1}), $V_{p}(\lambda )$ is $(k-NUC)$ for any
integer $k\geq 2$.
\end{proof}

\begin{theorem}
\label{Thm 2.3} For any $(1<p<\infty )$, the space $V_{p}(\lambda )$ has the
uniform Opial property.
\end{theorem}

\begin{proof}
Let $\varepsilon >0$ and $\varepsilon _{0}\in (0,\varepsilon )$. Also let $%
x\in X$ and $||x||\geq \varepsilon $. There exists $n_{1}\in \mathbb{N}$
such that 
\begin{equation*}
\dsum\limits_{n=n_{1}+1}^{\infty }\left( \frac{1}{\lambda _{n}}%
\dsum\limits_{i\in I_{n}}\left| x(i)\right| \right) ^{p}<\left( \frac{%
\varepsilon _{0}}{4}\right) ^{p}\text{.}
\end{equation*}
Hence we have 
\begin{equation*}
\left\| \dsum\limits_{i=n_{1}+1}^{\infty }x(i)e_{i}\right\| <\frac{%
\varepsilon _{0}}{4}\text{,}
\end{equation*}
where $e_{i}=(0,0,...,\overset{i\text{ }th}{1},0,0,...)$. Furthermore, we
have 
\begin{eqnarray*}
\varepsilon ^{p} &\leq &\dsum\limits_{n=1}^{n_{1}}\left( \frac{1}{\lambda
_{n}}\dsum\limits_{i\in I_{n}}\left| x(i)\right| \right)
^{p}+\dsum\limits_{n=n_{1}+1}^{\infty }\left( \frac{1}{\lambda _{n}}%
\dsum\limits_{i\in I_{n}}\left| x(i)\right| \right) ^{p} \\
&<&\dsum\limits_{n=1}^{n_{1}}\left( \frac{1}{\lambda _{n}}\dsum\limits_{i\in
I_{n}}\left| x(i)\right| \right) ^{p}+\left( \frac{\varepsilon _{0}}{4}%
\right) ^{p} \\
\varepsilon ^{p}-\left( \frac{\varepsilon _{0}}{4}\right) ^{p}
&<&\dsum\limits_{n=1}^{n_{1}}\left( \frac{1}{\lambda _{n}}\dsum\limits_{i\in
I_{n}}\left| x(i)\right| \right) ^{p}
\end{eqnarray*}
whence 
\begin{equation*}
\varepsilon ^{p}-\left( \frac{\varepsilon _{0}}{4}\right) ^{p}\leq
\dsum\limits_{n=1}^{n_{1}}\left( \frac{1}{\lambda _{n}}\dsum\limits_{i\in
I_{n}}\left| x(i)\right| \right) ^{p}\text{.}
\end{equation*}
Since $x_{m}(i)\rightarrow 0$ for $i=1,2,...$ \ , we choose any weakly null
sequences $\{x_{m}\}$ such that $\underset{m\rightarrow \infty }{\lim \inf }%
\left\| x_{m}\right\| \geq 1$. Then there exists $m_{0}\in \mathbb{N}$ such
that 
\begin{equation*}
\left\| \dsum\limits_{i=1}^{n_{1}}x_{m}(i)e_{i}\right\| <\frac{\varepsilon
_{0}}{4}
\end{equation*}
when $m>m_{0}$. Therefore, 
\begin{eqnarray*}
\left\| x_{m}+x\right\| &=&\left\| \dsum\limits_{i=1}^{n_{1}}\left(
x_{m}(i)+x(i)\right) e_{i}+\dsum\limits_{i=n_{1}+1}^{\infty }\left(
x_{m}(i)+x(i)\right) e_{i}\right\| \\
&\geq &\left\|
\dsum\limits_{i=1}^{n_{1}}x(i)e_{i}+\dsum\limits_{i=n_{1}+1}^{\infty
}x_{m}(i)e_{i}\right\| -\left\|
\dsum\limits_{i=1}^{n_{1}}x_{m}(i)e_{i}\right\| -\left\|
\dsum\limits_{i=n_{1}+1}^{\infty }x(i)e_{i}\right\| \\
&\geq &\left\|
\dsum\limits_{i=1}^{n_{1}}x(i)e_{i}+\dsum\limits_{i=n_{1}+1}^{\infty
}x_{m}(i)e_{i}\right\| -\frac{\varepsilon _{0}}{2}
\end{eqnarray*}
Moreover;

$\ \ \ \ \ \left\|
\dsum\limits_{i=1}^{n_{1}}x(i)e_{i}+\dsum\limits_{i=n_{1}+1}^{\infty
}x_{m}(i)e_{i}\right\| ^{p}$%
\begin{eqnarray*}
&& \\
\text{ \ \ \ \ \ \ \ \ \ } &=&\dsum\limits_{n=1}^{n_{1}}\left( \frac{1}{%
\lambda _{n}}\dsum\limits_{i\in I_{n}}\left| x(i)e_{i}\right| \right)
^{p}+\dsum\limits_{n=n_{1}+1}^{\infty }\left( \frac{1}{\lambda _{n}}%
\dsum\limits_{i\in I_{n}}\left| x_{m}(i)\right| \right) ^{p} \\
\text{ \ \ \ \ \ \ \ \ \ } &\geq &1+\varepsilon ^{p}-2\left( \frac{%
\varepsilon _{0}}{4}\right) ^{p}.
\end{eqnarray*}
Since $\left( 2\left( \frac{\varepsilon _{0}}{4}\right) ^{p}-1+\left(
1+\varepsilon _{0}\right) ^{p}\right) ^{\frac{1}{p}}>\varepsilon _{0}$ \ for 
$1<p<\infty ,$ we can choose$\ \varepsilon \geq \left( 2\left( \frac{%
\varepsilon _{0}}{4}\right) ^{p}-1+\left( 1+\varepsilon _{0}\right)
^{p}\right) ^{\frac{1}{p}}$ and we have

\begin{eqnarray*}
\left\Vert
\dsum\limits_{i=1}^{n_{1}}x(i)e_{i}+\dsum\limits_{i=n_{1}+1}^{\infty
}x_{m}(i)e_{i}\right\Vert &\geq &\left( 1+\varepsilon ^{p}-2\left( \frac{%
\varepsilon _{0}}{4}\right) ^{p}\right) ^{\frac{1}{p}} \\
&\geq &1+\varepsilon _{0}
\end{eqnarray*}%
Therefore, combining this result with the previous inequality, we get 
\begin{eqnarray*}
\left\Vert x_{m}+x\right\Vert &\geq &\left\Vert
\dsum\limits_{i=1}^{n_{1}}x(i)e_{i}+\dsum\limits_{i=n_{1}+1}^{\infty
}x_{m}(i)e_{i}\right\Vert -\frac{\varepsilon _{0}}{2} \\
&\geq &1+\varepsilon _{0}-\frac{\varepsilon _{0}}{2}=1+\frac{\varepsilon _{0}%
}{2}\text{.}
\end{eqnarray*}%
This means that $V_{p}(\lambda )$ has the uniform Opial property.
\end{proof}

From the Theorem \ref{Thm 2.2}, we get that $V_{p}(\lambda )$ is $(k-NUC)$.
Clearly $(k-NUC)$ Banach spaces are $(NUC)$, and $(NUC)$ implies property $%
(H)$ and reflexivity holds, \cite{Huff}. Also, Huff proved that $X$ is $%
(NUC) $ if and only if $X$ is reflexive and $(UKK)$ (see in \cite{Huff}).

On the other hand, it is well known that 
\begin{equation*}
(UC)\Rightarrow (kR)\Rightarrow (k+1)R\text{,}
\end{equation*}

\noindent and $(kR)$ spaces are reflexive and rotund, and it is easy to see
that 
\begin{equation*}
(k-NUC)\Rightarrow (kR).
\end{equation*}

By the facts presented in the introduction and the just above; we get the
following corollaries:

\begin{corollary}
The space $V_{p}(\lambda )$\ $(1<p<\infty )$\ is $(NUC)$ and then is
reflexive.
\end{corollary}

\begin{corollary}
The space $V_{p}(\lambda )$\ $(1<p<\infty )$\ is $(UKK)$.
\end{corollary}

\begin{corollary}
The space $V_{p}(\lambda )$\ $(1<p<\infty )$\ is $(kR)$.
\end{corollary}

\begin{corollary}
The space $V_{p}(\lambda )$\ $(1<p<\infty )$\ is rotund.
\end{corollary}

\end{document}